\pdfoutput=1
\RequirePackage{ifpdf}
\ifpdf 
\documentclass[pdftex]{sigma}
\else
\documentclass{sigma}
\fi

\numberwithin{equation}{section}

\newtheorem*{Problem*}{Open Problem}

\begin{document}

\renewcommand{\thefootnote}{}

\renewcommand{\PaperNumber}{005}

\FirstPageHeading

\ShortArticleName{Solution of an Open Problem about Two Families of Orthogonal Polynomials}

\ArticleName{Solution of an Open Problem \\ about Two Families of Orthogonal Polynomials\footnote{This paper is a~contribution to the Special Issue on Orthogonal Polynomials, Special Functions and Applications (OPSFA14). The full collection is available at \href{https://www.emis.de/journals/SIGMA/OPSFA2017.html}{https://www.emis.de/journals/SIGMA/OPSFA2017.html}}}

\Author{Walter VAN ASSCHE}

\AuthorNameForHeading{W.~Van Assche}

\Address{Department of Mathematics, KU Leuven, Belgium}
\Email{\href{mailto:walter.vanassche@kuleuven.be}{walter.vanassche@kuleuven.be}}
\URLaddress{\url{https://wis.kuleuven.be/analyse/members/walter}}

\ArticleDates{Received January 18, 2019; Published online January 27, 2019}

\Abstract{An open problem about two new families of orthogonal polynomials was posed by Alhaidari. Here we will identify one of them as Wilson polynomials. The other family seems to be new but we show that they are discrete orthogonal polynomials on a bounded countable set with one accumulation point at $0$ and we give some asymptotics as the degree tends to infinity.}

\Keywords{orthogonal polynomials; special functions; open problems}

\Classification{42C05; 33A45}

\renewcommand{\thefootnote}{\arabic{footnote}}
\setcounter{footnote}{0}

\section{Introduction}
Abdulaziz D.~Alhaidari \cite{alhaidari1} submitted the following open problem for the proceedings of the OPSFA-14 conference.

Using an algebraic method for solving the wave equation in quantum mechanics, Alhaidari and co-authors \cite{alhaid3, alhaid1,alhaid2,alhaid4,alhaid5} encountered two families of orthogonal polynomials on the real line. These polynomials are defined by their three-term recurrence relations and initial values. The weight functions (orthogonality measures), generating functions, orthogonality relations, etc.\ are yet to be derived analytically. See~\cite{alhaidari2} for more physical background.

The first family is a four parameter family of orthogonal polynomials $H_n^{(\mu,\nu)}(z;\alpha,\theta)$ given by the recursion
\begin{gather}
 \cos \theta H_n^{(\mu,\nu)}(z;\alpha,\theta) \nonumber \\
\qquad{} = \left( z \sin \theta \left[ \left(n+ \frac{\mu+\nu+1}{2}\right)^2 + \alpha \right]
 + \frac{\nu^2-\mu^2}{(2n+\mu+\nu)(2n+\mu+\nu+2)} \right) H_n^{(\mu,\nu)}(z;\alpha,\theta) \nonumber \\
\qquad\quad{} + \frac{2(n+\mu)(n+\nu)}{(2n+\mu+\nu)(2n+\mu+\nu+1)} H_{n-1}^{(\mu,\nu)}(z;\alpha,\theta) \nonumber \\
\qquad\quad{} + \frac{2(n+1)(n+\mu+\nu+1)}{(2n+\mu+\nu+1)(2n+\mu+\nu+2)} H_{n+1}^{(\mu,\nu)}(z;\alpha,\theta),\label{firstfam}
\end{gather}
with $0 \leq \theta \leq \pi$, $\mu,\nu > -1$ and $\alpha \in \mathbb{R}$. The initial values are $H_0^{(\mu,\nu)}(z;\alpha,\theta)=1$ and $H_{-1}^{(\mu,\nu)}(z;\alpha,\theta)=0$.

The second family is a three-parameter family of orthogonal polynomials $G_n^{(\mu,\nu)}(z;\sigma)$ which satisfies the recurrence
\begin{gather}\allowdisplaybreaks
 z G_n^{(\mu,\nu)}(z;\sigma) = \left( \big(\sigma+B_n^2\big) \left[ \frac{\mu^2-\nu^2}{(2n+\mu+\nu)(2n+\mu+\nu+2)} + 1 \right] \right.\nonumber\\
 \left. \hphantom{z G_n^{(\mu,\nu)}(z;\sigma) =}{} - \frac{2n(n+\nu)}{2n+\mu+\nu} - \frac{(\mu+1)^2}{2} \right) G_n^{(\mu,\nu)}(z;\sigma) \nonumber \\
\hphantom{z G_n^{(\mu,\nu)}(z;\sigma) =}{}- \big(\sigma+B_{n-1}^2 \big)\frac{2(n+\mu)(n+\nu)}{(2n+\mu+\nu)(2n+\mu+\nu+1)} G_{n-1}^{(\mu,\nu)}(z;\sigma) \nonumber \\
\hphantom{z G_n^{(\mu,\nu)}(z;\sigma) =}{}- \big(\sigma + B_n^2\big) \frac{2(n+1)(n+\mu+\nu+1)}{(2n+\mu+\nu+1)(2n+\mu+\nu+2)} G_{n+1}^{(\mu,\nu)}(z;\sigma),\label{secondfam}
\end{gather}
with $B_n= n+1 + \frac{\mu+\nu}{2}$, $\mu,\nu >-1$ and $\sigma \in \mathbb{R}$ and initial values $G_0^{(\mu,\nu)}(z;\sigma)=1$, $G_{-1}^{(\mu,\nu)}(z;\sigma)=0$.

\begin{Problem*} Due to the significance of the two families $H_n^{(\mu,\nu)}(z;\alpha,\theta)$ and $G_n^{(\mu,\nu)}(z;\sigma)$, we hope that experts in the field of orthogonal polynomials could study them, derive their analytic properties and write them in closed form $($e.g., in terms of hypergeometric functions$)$. The required properties for these polynomials include the weight functions, generating functions, asymptotics, orthogonality relations, Rodrigues-type formulas, forward/backward shift operators, zeros, etc.
\end{Problem*}

In order to identify the two families, it is convenient to switch to the monic polynomials and to look at the recurrence relation for the monic polynomials. One can then try to identify them with known families of orthogonal polynomials, e.g., by using the table of recurrence formulas in Chihara's book \cite{chihara}, or by using the information of hypergeometric orthogonal polynomials in the book of Koekoek, Lesky and Swarttouw~\cite{koekoek}, going through Chapter~18 on Orthogonal Polynomials in the Digital Library of Mathematical Function~\cite{refNIST, refDLMF}, or by comparing with the information available in CAOP\footnote{\url{http://www.caop.org/}} (Computer Algebra \& Orthogonal Polynomials). One can also try to use computer algebra to identify the polynomials from their recurrence relations, such as \texttt{rec2ortho}\footnote{\url{https://staff.fnwi.uva.nl/t.h.koornwinder/art/software/rec2ortho/}} of Koornwinder and Swarttouw or \texttt{retode}\footnote{\url{http://www.mathematik.uni-kassel.de/~koepf/Publikationen/\#down}} of Koepf and Schmersau~\cite{koepf}. These programs were designed for older versions of maple and do not handle all orthogonal polynomials in the Askey table. Recently Tcheutia \cite{tcheutia} extended the Maple implementation of \texttt{retode} to cover classical orthogonal polynomials on quadratic and $q$-quadratic lattices and was able to identify the second family \eqref{secondfam} as Wilson polynomials, confirming the analysis in Section \ref{sec3}.

\section{The first family of orthogonal polynomials} \label{sec2}

For the first family of orthogonal polynomials, the monic polynomials are given by $P_n(z) = H_n^{(\mu,\nu)}(z;\alpha,\theta)/k_n$, where
\begin{gather*} \frac{k_{n+1}}{k_n} = - \frac{\sin \theta \bigl[ \big(n+ \frac{\mu+\nu+1}{2}\big)^2+\alpha \bigr] (2n+\mu+\nu+1)(2n+\mu+\nu+2)}
 {2(n+1)(n+\mu+\nu+1)}. \end{gather*}
The recurrence relation then becomes
\begin{gather*} zP_n(z) = P_{n+1}(z) + b_n P_n(z) + a_n^2 P_{n-1}(z), \end{gather*}
with recurrence coefficients
\begin{gather*} a_n^2 = \frac{4n(n+\mu)(n+\nu)(n+\mu+\nu)}{\sin^2\theta \bigl[ \big(n\!+\! \frac{\mu{+}\nu{+}1}{2}\big)^2
\!+\!\alpha \bigr] \bigl[ \big(n\!+\! \frac{\mu{+}\nu{-}1}{2}\big)^2\!+\!\alpha \bigr] (2n\!+\!\mu\!+\!\nu\!+\!1)(2n\!+\!\mu\!+\!\nu)^2(2n\!+\!\mu\!+\!\nu\!-\!1)} \end{gather*}
and
\begin{gather*} b_n = \frac{1}{\sin \theta \bigl[ \big(n+ \frac{\mu+\nu+1}{2}\big)^2+\alpha \bigr]}
 \left( \cos \theta + \frac{\mu^2-\nu^2}{(2n+\mu+\nu)(2n+\mu+\nu+2)} \right). \end{gather*}
From this we can already deduce that for $\theta=\pi/2$
\begin{gather*} \lim_{\alpha \to \infty} \alpha^n P_n(z/\alpha) = P_n^{(\mu,\nu)}(z), \end{gather*}
where $P_n^{(\mu,\nu)}$ are the monic Jacobi polynomials. This follows by taking the limit for $\alpha \to \infty$ in the recurrence coefficients which, after appropriate scaling, gives the recurrence coefficients of the Jacobi polynomials.
Another useful observation is that
\begin{gather*} \lim_{n \to \infty} a_n^2 = 0, \qquad \lim_{n \to \infty} b_n = 0, \end{gather*}
and in fact $a_n^2 = \mathcal{O}\big(1/n^4\big)$ and $b_n = \mathcal{O}\big(1/n^2\big)$. This implies that the Jacobi matrix
\begin{gather*} J = \begin{pmatrix} b_0 & a_1 & 0 & 0 & 0 & \cdots \\
 a_1 & b_1 & a_2 & 0 & 0 & \cdots \\
 0 & a_2 & b_2 & a_3 & 0 & \cdots \\
 0 & 0 & a_3 & b_3 & a_4 & \\
 \vdots & \vdots & & \ddots & \ddots & \ddots
 \end{pmatrix} \end{gather*}
is a compact operator, and in fact it is a trace class operator. This implies that the spectrum of~$J$ is a countable set $\{x_k, k \in \mathbb{N}\}$ with one accumulation point at $0$, hence $\lim\limits_{k \to \infty} x_k = 0$. The trace class condition implies that $\sum\limits_{k=0}^\infty |x_k|$ is finite. Consequently the orthogonality measure for the first family is a discrete measure supported on this countable set. See~\cite{wva} for more information on compact operators and orthogonal polynomials.

The asymptotic behavior is given by
\begin{theorem} \label{thm:Pasy}
There exists an entire function $Q$ such that
\begin{gather} \label{Pasy}
 \lim_{n \to \infty} z^n P_n(1/z) = Q(z),
\end{gather}
uniformly on compact subsets of $\mathbb{C}$. The function $Q$ has infinitely many zeros at the points $\{1/x_k, k \in \mathbb{N} \}$.
\end{theorem}

\begin{proof}
We introduce the reversed polynomials $Q_n(z) = z^n P_n(1/z)$. We have $Q_n(0)=1$ since the $P_n$ are monic polynomials.
The three-term recurrence relation then becomes
\begin{gather} \label{Qrec}
 Q_n(z) = Q_{n+1}(z) + b_n z Q_n(z) + a_n^2 z^2 Q_{n-1}(z),
\end{gather}
from which we easily find
\begin{gather*} Q_{k+1}(z) - Q_k(z) = -b_kz Q_k(z) - a_k^2 z^2 Q_{k-1}(z). \end{gather*}
Summing from $k=0$ to $n-1$ then gives
\begin{gather}
 Q_n(z) = Q_0(z) - \sum_{k=0}^{n-1} b_k zQ_k(z) - \sum_{k=1}^{n-1} a_k^2 z^2 Q_{k-1}(z) \nonumber \\
 \hphantom{Q_n(z)}{} = 1 - \sum_{k=0}^{n-1} b_k z Q_k(z) - \sum_{k=0}^{n-2} a_{k+1}^2 z^2 Q_k(z).\label{Q}
\end{gather}
Let $K$ be a compact set in $\mathbb{C}$, then there exists an $R >0$ such that $|z| \leq R$ for $z \in K$ and from~\eqref{Q} we find
\begin{gather*} |Q_n(z)| \leq 1 + \sum_{k=0}^{n-1} \big(|b_k| R + a_{k+1}^2 R^2\big) |Q_k(z)|. \end{gather*}
We can then use the discrete version of Gronwall's inequality (see \cite[p.~440]{wvaNATO}) to find that
\begin{gather*} |Q_n(z)| \leq \exp \left( \sum_{k=0}^{n-1} \big(|b_k|R+a_{k+1}^2 R^2\big) \right), \end{gather*}
so that uniformly for $z \in K$
\begin{gather*} |Q_n(z)| \leq M = \exp \left( R \sum_{k=0}^\infty |b_k| + R^2 \sum_{k=0}^\infty a_{k+1}^2 \right). \end{gather*}
We can then use Lebesgue's dominated convergence theorem and take the limit $n \to \infty$ in \eqref{Q} to find
\begin{gather} \label{Qasy}
 \lim_{n \to \infty} Q_n(z) = 1 - \sum_{k=0}^\infty \big(b_kz+a_{k+1}^2 z^2\big) Q_k(z) := Q(z) ,
\end{gather}
and the sum converges uniformly for $z \in K$. The limit function $Q$ is therefore an entire function, and its zeros are limits of the zeros of $Q_n$, which in turn are $\{ 1/x_{n,k}, 1 \leq k \leq n\}$, where $x_{n,k}$ are the zeros of $P_n$.
\end{proof}

Theorem \ref{thm:Pasy} gives the existence of an entire function $Q$ for which \eqref{Pasy} holds. The formula~\eqref{Qasy} is not very convenient to describe the limit function, since it is in terms of the polynomials $Q_k$ that we are investigating. In order to find more information on $Q$, we can write
\begin{gather*} Q_n(z) = \sum_{k=0}^n c_{n,k} z^k, \qquad Q(z) = \sum_{k=0}^\infty c_k z^k, \end{gather*}
and then $c_{n,0}=1=c_0$ and
\begin{gather*} c_k = \lim_{n \to \infty} c_{n,k}. \end{gather*}
From the recurrence relation \eqref{Qrec} we find
\begin{gather*} c_{n,k} = c_{n+1,k} + b_n c_{n,k-1} + a_n^2 c_{n-1,k-2}, \end{gather*}
which, in principle, allows to compute the coefficients $c_{n,k}$ recursively. The first few coefficients are
\begin{gather*} c_{n,1} = - \sum_{k=0}^{n-1} b_k, \quad c_{n,2} = \sum_{k=1}^{n-1} \sum_{j=0}^{k-1} b_kb_j - \sum_{k=1}^{n-1} a_k^2, \end{gather*}
from which we find
\begin{gather*} c_1 = - \sum_{k=0}^\infty b_k, \quad c_2 = \sum_{k=1}^\infty \sum_{j=0}^{k-1} b_kb_j - \sum_{k=1}^\infty a_k^2 . \end{gather*}
These sums are all absolutely convergent. Their explicit value depends on the parameters $\mu$, $\nu$, $\alpha$, $\theta$
and involves the expression
\begin{gather*} \psi \left( \frac{\mu+\nu+1}{2} - \sqrt{-\alpha} \right) - \psi \left( \frac{\mu+\nu+1}{2} + \sqrt{-\alpha} \right), \end{gather*}
where $\psi(z) = \Gamma'(z)/\Gamma(z)$ is the Psi function \cite[Section~5.2]{refNIST, refDLMF}, so that a distinction between the cases $\alpha>0$ and $\alpha < 0$ may be needed. No further analysis of the function $Q$ has been done.

\section{The second family of orthogonal polynomials} \label{sec3}

For the second family the monic polynomials are $P_n(z) = G_n^{(\mu,\nu)}(z;\sigma)/k_n$, where
\begin{gather*} \frac{k_n}{k_{n+1}} = -\big(\sigma+B_n^2\big) \frac{2(n+1)(n+\mu+\nu+1)}{(2n+\mu+\nu+1)(2n+\mu+\nu+2)}. \end{gather*}
The recurrence coefficients for the monic polynomials are then given by
\begin{gather*} a_n^2 = \big(\sigma+B_{n-1}^2\big)^2 \frac{4n(n+\mu)(n+\nu)(n+\nu+\mu)}{(2n+\mu+\nu-1)(2n+\mu+\nu)^2(2n+\mu+\nu+1)}, \end{gather*}
and
\begin{gather*} b_n = \big(\sigma+B_n^2\big) \left( \frac{\mu^2-\nu^2}{(2n+\mu+\nu)(2n+\mu+\nu+2)} + 1 \right) - \frac{2n(n+\nu)}{2n+\mu+\nu} - \tfrac12 (\mu+1)^2 . \end{gather*}
Again there is a limit transition to Jacobi polynomials:
\begin{gather*} \lim_{\sigma \to \infty} \sigma^{-n} P_n(\sigma z) = P_n^{(\mu,\nu)}(z-1), \end{gather*}
where $P_n^{(\mu,\nu)}$ are the monic Jacobi polynomials.
The recurrence coefficients have the asymptotic behavior
\begin{gather*} a_n^2 = \tfrac14{n^4} + \mathcal{O}\big(n^3\big), \qquad b_n = n^2 + \mathcal{O}(n), \end{gather*}
so that they are unbounded. The spectrum of the Jacobi matrix (and the support of the orthogonality measure for the polynomials) is therefore
unbounded. It was noted by Yutian Li \cite{Li} that the recurrence coefficients correspond to a special case of the Wilson polynomials
$W_n(x;a,b,c,d)$, which are on top of the Askey table \cite[Section~9.1]{koekoek}. The identification is
\begin{gather*} G_n^{(\mu;\nu)}(z,\sigma) = \frac{W_n(z/2;a,b,c,d)}{(a+b)_n(a+d)_n}, \end{gather*}
where the parameters $a$, $b$, $c$, $d$ are given by
\begin{gather*} a=\tfrac12{(\mu+1)} = c, \qquad b= \tfrac12{(\mu+1)} +s, \qquad d = \tfrac12{(\mu+1)} -s, \end{gather*}
with $s=\sqrt{-\sigma}$.

\subsection*{Acknowledgements}
WVA is supported by EOS project PRIMA 30889451 and FWO research project G.086416N.

\pdfbookmark[1]{References}{ref}
\LastPageEnding

\end{document}